\documentclass[12pt,a4paper]{article} %WC to t??!
\usepackage{amsmath,amssymb,amsthm,amsfonts,amscd,euscript,verbatim, t1enc, newlfont}
\usepackage{hyperref}
\usepackage{cancel} 

\hypersetup{breaklinks}
\theoremstyle{definition}
\newtheorem{theo}{Theorem}[subsection]

\newtheorem{pr}[theo]{Proposition}

 \newtheorem{lem}[theo]{Lemma}

\theoremstyle{remark}
\newtheorem{rema}[theo]{Remark}

\newtheorem{defi}[theo]{Definition}

\newcommand\cu{\underline{C}}
\newcommand\du{\underline{D}}
\newcommand\eu{\underline{E}}

\newcommand\bu{\underline{B}}
\newcommand\hw{\underline{Hw}}

\newcommand\ab{\underline{\operatorname{Ab}}}
\DeclareMathOperator\prli{\varprojlim}
\DeclareMathOperator\inli{\varinjlim}
\newcommand\spe{\operatorname{Spec}}
\newcommand\afo{\mathbb{A}^1}

\newcommand\p{{\mathbb{P}}}
\newcommand\z{{\mathbb{Z}}}
\newcommand\q{{\mathbb{Q}}}
\newcommand\cp{\mathcal{P}}
\newcommand\zop{{\mathbb{Z}[1/p]}}
\newcommand\ns{\{0\}}

\newcommand\sv{\operatorname{SmVar}}
\newcommand\spv{\operatorname{SmPrVar}}
\newcommand\chow{\operatorname{Chow}}
\newcommand\chowe{\chow^{eff}}

\newcommand\dmger{\dm^{eff}_{gm,R}}

\newcommand\dmge{DM^{eff}_{gm}}
\newcommand\dmgm{DM_{gm}}

\newcommand\dm{\operatorname{DM}}
\newcommand\dmer{{\dm^{eff}_{R}}}
\newcommand\dmez{\dm^{eff}_{\z}{}}

\newcommand\dmerpp{\dm^{eff}_{R'}{}'}

\newcommand\dmb{\dm^{bir}_{R}{}}

\newcommand\mg{M_{gm}}

\newcommand\dmr{\dm_{R}}
\newcommand\mgr{{\mathcal{M}}_R}
\newcommand\mgz{{\mathcal{M}}_{\z}}

\newcommand\smc{\operatorname{SmCor}}
\newcommand\sscr{\operatorname{Sh}_{Nis}(\smc,R)}
\newcommand\sscrp{\operatorname{Sh}_{Nis}(\smc,R')}
\newcommand\pscr{\operatorname{PreSh}_{Nis}(\smc,R)}
\newcommand\zmv{\operatorname{Zar-MV}}
\newcommand\nmv{\operatorname{Nis-MV}}

\newcommand\smcrpl{ \smc_R^{\oplus}} %\newcommand\smcrpl{ \smc_R^{\coprod}}??
\newcommand\id{\operatorname{id}}

\newcommand\mgrp{\mathcal{M}_{R'}}

\newcommand\chower{\chow^{eff}_R}
\newcommand\chowerp{\chow^{eff}_{R'}}
\newcommand\chowr{\chow_R}
\newcommand\mgb{\mathcal{M}^{bir}_{R}}
\newcommand\dmgb{DM^{bir}_{gm,R}}
\newcommand\chowb{\chow^{bir}_R}

\newcommand\wchows{w_{\chow}^{eff}}

\newcommand\wchow{w_{\chow}}

\newcommand\wbir{w_{R}^{bir}}
\newcommand\rtr{{R}_{tr}}

\newcommand\obj{\operatorname{Obj}}
 \newcommand\lan{\langle}
\newcommand\ra{\rangle}

\newcommand\perpp{{}^{\perp}}

\DeclareMathOperator\kar{\operatorname{Kar}}
\DeclareMathOperator\co{\operatorname{Cone}}

\numberwithin{equation}{subsection}

\begin{document}

\title{On Chow weight structures without projectivity and resolution of singularities } 
 \author{Mikhail V. Bondarko, David Z. Kumallagov   \thanks{ %%!!!!
 The work of the first author on sections  1, 2.2, and 3.1 of this paper was supported by the Russian Science Foundation grant No. 16-11-10200.}}

\maketitle
\begin{abstract}
In this paper certain Chow weight structures on the "big" triangulated motivic categories  $\dmer\subset \dmr$ are defined in terms of motives of all smooth varieties over the base field. This definition allows studying basic properties of these weight structures without applying resolution of singularities; thus we don't have to assume that the coefficient ring  $R$ contains $1/p$ in the case where the characteristic $p$ of the base field is positive. Moreover, in the case where $R$ satisfies the latter assumption  our weight structures %coincide with the weight structures 
 are "compatible" with the weight structures that were defined  in previous papers in terms of Chow motives; it follows that a motivic complex has non-negative weights if and only if its positive Nisnevich hypercohomology vanishes. The results of this article yield certain Chow-weight filtration (also) on $p$-adic cohomology of motives and smooth varieties. 
\end{abstract}

\tableofcontents

\section*{Introduction}

This paper is dedicated to the study of a new definition of Chow weight structures for Voevodsky motives over an arbitrary perfect field $k$ of characteristic  $p$. This definition  does not depend on the existence of nice compactifications for smooth varieties (and other resolution of singularities results); this allows treating  $R$-linear versions of these weight structures  (on the triangulated categories $\dmer\subset \dmr$) also in the case where  $p$ is positive and not invertible in the coefficient ring $R$.

Now recall that Chow weight structures yield %motivic certain 
 analogues of Deligne's weights  (as described for mixed Hodge structures in  \cite{delh2} and for mixed \'etale sheaves in \cite{delw2}) for various triangulated categories of Voevodsky motives. For motives over a field certain Chow weight structures were described in  \cite{bws} and \cite{bzp} (the latter paper treated the case  $p>0$). In these articles %papers %работах Бондарко 
 the fact that the categories of geometric (i.e., compact) motives are generated by their subcategories of Chow motives (i.e., by $\chowe$ and $\chow$, respectively) was applied. It yields the existence of bounded weight structures on the categories $\dmge$ and $\dmgm$ of geometric motives; their hearts consist of the corresponding categories of Chow motives. Moreover, these weight structures can be extended to the corresponding "big" motivic categories (that are compactly generated by their subcategories of geometric motives; cf. Proposition 1.7 and \S2.1 of \cite{binters} or \S2.3 of  \cite{bkl}). In contrast to Deligne's weights, this gave Chow weight structures for 
$\zop$-linear motives (in the case $p>0$ we set $\zop=\z$); respectively, %they yield
 there exist  weight filtrations and spectral sequences for all $\zop$-linear (co)homology of motives. Moreover, the theory of weight structures gives non-trivial functoriality properties of these matters. 
Similarly to several other properties of motives, the construction of weight structures in the aforementioned papers relied on the resolution of singularities (i.e., on the Hironaka theorem in the case $p=0$ and on the Gabber's resolution of singularities in the setting of $\zop$-linear motives for $p>0$). 

In the current paper a new construction method is applied; it gives a certain Chow weight structure  $\wchows$  "directly" %???
 on the category  $\dmer=\dmer(k)$ of unbounded $R$-linear motivic complexes,  %(и ее локализациях $\dmelam$): 
 where $R$ is an arbitrary (commutative unital) coefficient ring; this weight structure is generated by motives of all smooth varieties over $k$.\footnote{Weight structures generated by sets of compact objects were constructed in   \cite{paucomp};  their properties were studied in detail in \cite{bpure}. It is also worth noting that in several papers  (starting from \cite{wild}) J. Wildeshaus has used motives of non-proper $k$-varieties for the  construction  %and studied 
 of certain motives in the heart of the Chow weight structure.} %non-smooth???
 This definition (in contrast to earlier ones based on smooth projective varieties only) does not depend on any resolution of singularities results; so it also "works fine" for $\z$-linear motives over any perfect field $k$ of characteristic $p>0$; see Remark  \ref{rnewoldchow}(4) below. A disadvantage of this method is that it does not yield that $\wchows$ is generated by Chow motives and restricts to the subcategory $\dmger(k)$ of geometric motives. Still we successfully establish several other properties of $\wchows$ that are similar to the main properties of the Chow weight structures defined earlier.  In particular, we prove that  $\wchows$ can be naturally extended to a weight structure $\wchow$ on the so-called stable motivic category $\dmr$ (that contains $\dmer$).

Now we describe the contents of the paper; some  more information of this sort can also be found in the beginnings of sections.

In \S\ref{sprel} we give some categorical notation and definitions, and recall the basics of the theory of weight structures (on compactly generated triangulated categories). The most complicated part of the section is the recollection of the properties of the category $\dmer$ (or $R$-linear motivic complexes) for an arbitrary coefficient ring $R$; since the existing literature is mostly dedicated to the case  $R\subset \q$, we are forced to compare different definitions of $\dmer$. Note here that in the ("basic") cases  $R=\z$ and $R=\zop$ those arguments of our papers that concern "general" properties of motives can be significantly simplified (for instance, one may apply the results of  \cite{deg} and \cite{kel}).

In \S\ref{swchows} our main weight structure  $\wchows$ on the category $\dmer$ is defined; we prove some of its properties. In particular, we prove that 
 $\wchows$ is generated by $R$-linear Chow motives if either $p$ is zero or $p$ is invertible in $R$. Thus we obtain the compatibility of the Chow weight structures defined earlier with $\wchows$; it follows that a motivic complex has non-negative weights if and only if its positive Nisnevich hypercohomology vanishes.
Moreover, all the scalar extension functors  $-\otimes^{mot}_R{R'}$ (where $R'$ is a commutative unital $R$-algebra) are weight-exact. 

In  \S\ref{swchowo} we prove that  $\wchows$ naturally "induces"  certain weight structures on the categories of stable and birational motives (i.e., we consider a compactly generated category $\dmr$ on which the action of the Tate twist  $-\lan 1 \ra=-(1)[2]$ is invertible, and the "birational" localization $\dmer/\dmer\lan 1\ra$). These statements are similar to the corresponding properties of Chow weight structures that were established in previous papers as well. 

The authors are deeply grateful to the referee for his very helpful comments.

\section{Preliminaries}\label{sprel}

In \S\ref{snotata}   we introduce some  definitions and notation that are mostly related to triangulated categories; we also prove an easy lemma.

%вряд ли является новым, но???

In \S\ref{sws} we recall the basics of the theory of weight structures. % that we will need below.

In \S\ref{smot} we briefly recall the basics on (unbounded) Voevodsky motivic complexes with coefficients in an arbitrary ring $R$. %Our main source for the category   $\dmer$ is \cite{bev}; we
 We  also prove some properties of the motivic extension of scalars functors.

\subsection{Categorical definitions and notation}\label{snotata}

%\qquad 
\begin{itemize}

\item Given a category $\bu$ and  $M,N\in\obj \bu$, we say that $M$ is a {\it
retract} of $N$ %(and $Y$ is {\it coretract} of $X$)
 if $\id_M$ can be %factorized as $X\stackrel{i}{\to} Y\stackrel{p}{\to}X$
 factored through $N$ (recall that if $\bu$ is triangulated then $M$ is a retract of $N$ if and only if $M$ is its direct summand).

\item A subcategory $\du$ of  $\bu$ is said to be  {\it Karoubi-closed}
  in $\bu$ if it contains all $\bu$-retracts of its objects.

\item The full subcategory $\kar_{\bu}(\du)$ of $\bu$ whose objects are all $\bu$-retracts of objects of $\du$  will be called the {\it Karoubi-closure} of $\du$ in $\bu$. It is easily seen that $\kar_{\bu}(\du)$ is Karoubi-closed in   $\du$; if  $\bu$ and $\du$ are additive then  $\kar_{\bu}(\du)$ is additive as well. 

\item We will say that an additive category $\du$ is  {\it Karoubian} if  any its idempotent endomorphism is isomorphic to the composition of a retraction and a coretraction of the type $M\bigoplus N\to M\to M\bigoplus N$.

\item The symbol $\cu$ below will always denote some triangulated category.
For a given class  $\cp\subset \obj \cu$ we will write $\lan \cp\ra$ for the smallest full  Karoubi-closed triangulated subcategory $\du$ of $\cu$ such that $\cp\subset \obj \du$.

\item For any  $A,B,C \in \obj\cu$ we will say that $C$ is an {\it extension} of $B$ by $A$ if there exists a distinguished triangle $A \to C \to B \to A[1]$.
A class $\cp\subset \obj \cu$ is said to be  {\it extension-closed}    if it is closed with respect to extensions and contains $0$.

\item The  smallest extension-closed Karoubi-closed  class $\cp'\subset \obj \cu$ containing $\cp$ will be called the {\it envelope} of $\cp$.

\item For $M,N\in \obj \cu$ we will write $M\perp N$ if $\cu(M,N)=\ns$. For $D,E\subset \obj \cu$ we write $D\perp E$ if $M\perp N$ for all $M\in D,\ N\in E$.

\item Given $\cp\subset\obj \cu$ we  will write $\cp^\perp$ for the class $$\{N\in \obj \cu:\ M\perp N\ \forall M\in \cp\}.$$
%Sometimes we will denote by $D^\perp$ the corresponding full subcategory of $\cu$. 
Dually, $\perpp \cp=\{M\in \obj \cu:\ M\perp N\ \forall N\in \cp\}$.

%???\item Мы будем называть функторы, сохраняющие структуры триангулированной категории, {\it точными}.
\item\label{idloc} Assume that $\cu$ is closed with respect to (small) coproducts (we only consider small coproducts in this paper). 
 For $\du\subset \cu$ ($\du$ is a triangulated category that may be equal to $\cu$)
  one says that %some  objects $D_i$ of $ \du$ () a class  
	$\cp$ generates $\du$ {\it as a localizing subcategory} of $\cu$ if  
$\du$ is the smallest full strict triangulated subcategory of $\cu$ that contains $\cp$ %and has coproducts. 
and is closed with respect to  $\cu$-coproducts.

\item $M\in\obj\cu$ is said to be {\it compact} if the functor  $\cu(M,-):\cu\to \ab$ respects coproducts.  

\item $\cu$ is said to be {\it compactly generated} if it is generated by a set of compact objects as its own localizing subcategory. 
\end{itemize}

We will sometimes need the following properties of compactly generated triangulated categories.

\begin{lem}\label{lcg}
Assume that the category  $\cu$ is %compactly generated by the object set of its small 
 generated as its own localizing subcategory by  objects of its triangulated subcategory $\cu'$, %стабильной относительно сдвигов??
the objects of $\cu'$ are compact in $\cu$, and  $F:\cu\to \du$ is an exact functor (so, the category $\du$ is triangulated) that respects coproducts.

1. If $\cu'$ is small then there exists an exact functor $G:\du\to \cu$ right adjoint to $F$. 

2. Assume that the restriction of  $F$ to $\cu'$ is a full embedding that sends objects of  $\cu'$ into compact objects of  $\du$. Then  $F$ is fully faithful. 

Moreover, if the class $F(\obj \cu')$ generates  $\du$  as its own localizing subcategory then  $F$ is an equivalence of categories. 

3. Let $E$ be a set  of objects of $ \obj \cu'$ , %$\eu'=\lan D\ra$, а 
 and denote by $\eu$  the  localizing subcategory of  $\cu$ generated by $E$. Then the localization $\cu/\eu$ exists  (i.e., the morphism classes in  $ \cu/\eu$ are sets), and it is closed with respect to coproducts. Moreover, the localization functor  $L:\cu\to \cu/\eu$ %переводит компактные объекты в компактные,  %Кроме того,  
respects coproducts and induces an equivalence of $\kar(\kar(\cu')/\eu')$ with the full subcategory of compact objects of  $\cu/\eu$, and the class $L(\obj \cu')$ generates  $\cu/\eu$ as its own localizing subcategory. 
\end{lem}
\begin{proof}
1. See Theorem  8.4.4 and Lemma 5.3.6 of \cite{neebook}. % 8.3.3

2. Since the objects of $\cu'$ are compact in $\cu$ and their images are compact in  $\du$, the class $C_1$ of those $N\in \obj\cu$ such that the maps $\cu(M,N)\to \du(F(M),F(N))$ are bijective for all  $M\in \obj \cu'$, is closed with respect to coproducts. %?????? It 
 $C_1$ is also shift-invariant (since $\obj \cu'[1]=\obj\cu'$); hence %??!!
it is also extension-closed. Since $C_1$  contains $\obj \cu'$, %this class 
we obtain $C_1=\obj\cu$. Next, the class $C_2$ of those  $M\in \obj\cu$ such that the maps $\cu(M,N)\to \du(F(M),F(N))$ are bijective for all  $N\in \obj \cu$ is obviously closed with respect to coproducts, shifts, and extensions (here we apply the assumption that  $F$ respects coproducts once again). As we have just proved, %this latter class 
 $C_2$ contains $\obj \cu'$;  thus it coincides with $\obj \cu$, i.e., $F$ is fully faithful.

The second part of the assertion easily follows from Proposition  1.1.4 of \cite{bpure}.

3. These statements also follow from the results of  \cite{neebook} easily; indeed, the easy arguments described in the proof of \cite[Proposition 4.3.1.3(III.1--2)]{bos} demonstrate that they follow from  Theorem 8.3.3,  Proposition 9.1.19, and Theorem 4.4.9 of \cite{neebook}.
\end{proof}
%\end{comment}

Now we introduce some "geometric" notation.

Our base field will be denoted by $k$; we assume that it is perfect (and fixed). We will write $p$ for its characteristic  ($p$ may equal $0$). Moreover, if $p=0$ then the %notation 
 symbol  $\zop$ will denote the ring $\z$.
  %Обозначим через 
 
$\sv$ is the set of smooth (not necessarily connected) $k$-varieties.

We will use the notation $R$ for the "main" coefficient ring for motives in this paper; $R$ will always be a commutative associative unital ring.

\subsection{Weight structures}\label{sws}

Recall that  $\cu$ will always denote some triangulated category in the current paper.

\begin{defi}\label{dwstr}

 A couple of subclasses $\cu_{w\le 0}$ and $ \cu_{w\ge 0}\subset\obj \cu$ %(of {\it $w$-negative} and {\it $w$-positive} objects, respectively)
will be said to define a {\it weight structure} $w$ for a triangulated category  $\cu$ if they  satisfy the following conditions.

(i) $\cu_{w\le 0}$ and $\cu_{w\ge 0}$ are %additive and 
Karoubi-closed in $\cu$
(i.e., contain all $\cu$-retracts of their elements).

(ii) {\bf Semi-invariance with respect to translations.}

$\cu_{w\le 0}\subset \cu_{w\le 0}[1]$ and $\cu_{w\ge 0}[1]\subset
\cu_{w\ge 0}$.

(iii) {\bf Orthogonality.}

$\cu_{w\le 0}\perp \cu_{w\ge 0}[1]$.

(iv) {\bf Weight decompositions}.

 For any $M\in\obj \cu$ there
exists a distinguished triangle
%\begin{equation}\label{wd}
$$LM\to M\to RM {\to} LM[1]$$
%\end{equation} 
such that $LM\in \cu_{w\le 0} $ and $ RM\in \cu_{w\ge 0}[1]$.
\end{defi}

We will also need the following definitions.

\begin{defi}\label{dwso}
%Let $i,j\in \z$
\begin{enumerate}
\item
The full subcategory  $\hw\subset \cu$ whose object class is $\cu_{w=0}=\cu_{w\ge 0}\cap \cu_{w\le 0}$ 
 is called the {\it heart} of  $w$.

\item For $i\in \z$ we will use the notation $\cu_{w\ge i}$ (resp. $\cu_{w\le i}$, resp. $\cu_{w= i}$)  for the class  $\cu_{w\ge 0}[i]$ (resp. $\cu_{w\le 0}[i]$,  $\cu_{w= 0}[i]$).

\item We will say that a weight structure  $w$ is {\it generated} by a class $\cp\subset \obj \cu$ if $\cu_{w\ge 0}=(\cup_{i>0} \cp[-i])\perpp$. 

\item\label{id5} Let $\cu'$  be a triangulated category endowed with a weight structure  $w'$; let $F:\cu\to \cu'$ be an exact functor.

We will say that $F$ is {\it left weight-exact} (with respect to $w,w'$) if it maps $\cu_{w\le 0}$ into $\cu'_{w'\le 0}$; it will be called {\it right weight-exact} if it sends $\cu_{w\ge 0}$ into $\cu'_{w'\ge 0}$. $F$ is said to be {\it weight-exact} if it is both left and right weight-exact. 

\end{enumerate}\end{defi}

A collection of basic properties of weight structures can be found in \S2 of \cite{bpure}.\footnote{These statements were actually proved in \cite{bws} (cf. also Remark 1.2.3(4) of \cite{bonspkar}); yet in that  paper somewhat distinct notation for weight structures was used.} 
In the current paper we will apply the following statements. 

\begin{pr}\label{pgenw}

I. Let $w$ be a weight structure on $\cu$.

1. Then  $\cu_{w\le 0}=\perpp\cu_{w\ge 1}$ and  $\cu_{w\ge 0}=\cu_{w\le 1}\perpp$. Thus if $w$ is generated by a class $\cp$ then  $\cp\subset \cu_{w\ge 0}$.

2. For each $i\in \z$ the classes $\cu_{w\ge i}$ and $\cu_{w\le i}$ are extension-closed (hence they are additive). 

II. Assume that $\cu$ is compactly generated. %замкнута относительно (малых) копроизведений.

1. Let  $\cp$ be a set of compact objects of  $\cu$. Then there exists (a unique) weight structure on $\cu$ that is generated by  $\cp$. % \footnote{Кроме того, для любых $\cu$ и $\cp$ существует не более одной весовой структуры, порожденной $\cp$ в $\cu$.} %существует.
%существует (единственная) весо

2. Assume that an exact functor  $F:\cu\to \cu'$ respects coproducts, $w$ is a weight structure on  $\cu$ that is generated by some class $\cp\subset \obj \cu$,
and $w'$ is a weight structure on  $\cu'$. Then $F$ is left weight-exact if and only if   $F(\cp)\subset \cu_{w\le 0}$.

3. Assume in addition that  $F$ is surjective on objects.  Then $F$ is weight-exact if and only if $w'$ is generated by $F(\cp)$.
\end{pr}
\begin{proof}
I. See \cite[Proposition 2.1.4(3), Remark 2.1.5(1)]{bpure}. %{bws}; см. обзор?

II.1. Immediate from Theorem 5 of \cite{paucomp};  cf. also  \cite[%замечание
Remark 4.2.2(1)]{bpure}. 

2. See Remark 2.1.5(3,4) of ibid. 

3. See Remark 2.1.5(7)  of loc. cit.
\end{proof}

\subsection{On Voevodsky motivic complexes}\label{smot}

We briefly recall the basics of the theory of  ($R$-linear unbounded) Voevodsky motivic complexes. The case of an arbitrary (associative commutative unital)  $R$
was not treated in detail in the literature; yet it is not much different from the "main" case  $R=\z$ (or $R$ being a localization of  $\z$; cf. \S5 of  \cite{kel}).

%{cdint}?!
$R$-linear unbounded motivic complexes are defined for an arbitrary %(associative commutative unital) ring 
 $R$; see  \cite[\S2.3]{bev}.\footnote{Recall that bounded above motivic complexes were considered in Lecture 14 of \cite{vbook}.}
We start from the description of  the category $\dmer$  given in loc. cit. One takes the additive category $\smc$ of Voevodsky smooth correspondences  (the notation is taken from \cite{1}); so, $\obj \smc=\sv$ and the morphisms in $\smc$ are algebraic analogues of multi-valued functors.  $\pscr$ will denote the abelian category of additive contravariant functors from $\smc$ into $R$-Mod. %-\modd$. %Легко видеть, что из   \cite[Theorem 13.1]{vbook} следует абелевость этой категории. 
 
For $X\in \sv$ we will use the notation  $\rtr(X)$ for the functor %пучок (см. там же) %??
 $Y\mapsto \smc(Y,X)\otimes_{\z}R$; this is an object of $\pscr$ that we will also consider as a complex (and so, as an object of $D(\pscr)$) whose non-zero term is in degree  $0$. The object  $\rtr(X)$ is certainly compact in $D(\pscr)$; %легко видеть, что 
	 hence \cite[Theorem 8.3.3]{neebook} implies that the set of all  $\rtr(X)$  generates $D(\pscr)$ as its own localizing subcategory.

$\dmer$ is defined as the Verdier quotient of  $D(\pscr)$ by the localizing subcategory generated by the union of two sets of complexes that we will now describe. 
%the two-term complexes  $\rtr(\afo\times X)\stackrel{pr_X}{\to} \rtr(X)$ for $X\in \sv$ (here the morphism $pr_X$ comes from the projection $\afo\times X\to X$; we will use the notation $HI$ for the set of this complexes)  along with the complexes 
The first of these sets  will be denoted by $HI$; its elements are two-term complexes  $\rtr(\afo\times X)\stackrel{pr_X}{\to} \rtr(X)$ for $X\in \sv$ (here the morphism $pr_X$ comes from the projection $\afo\times X\to X$).
The second set will be denoted by $ \zmv$; its elements are complexes  
\begin{equation}\label{enisc}
\rtr(W)\xrightarrow{\begin{pmatrix}-\rtr(k) \\ \rtr(g) \end{pmatrix}} \rtr(Y)\bigoplus \rtr(V)\xrightarrow{\begin{pmatrix} \rtr(f) & \rtr(j) \end{pmatrix}} \rtr(X)\end{equation}
 corresponding to all Cartesian squares
\begin{equation}\label{enis} %$$\
\begin{CD}
 W@>{k}>>Y\\
@VV{g}V@VV{f}V \\
V@>{j}>>X
\end{CD}\end{equation}
%$$
of smooth varieties such that  connecting morphisms are open embeddings  (hence  $W=V\cap Y$)   and $Y\sqcup V$ covers $X$. %We will use the notation  $ \zmv$ for the set of all these complexes. 
 We will write  $L'_R$ for the corresponding localization functor and note that this functor respects coproducts and compact objects according to Lemma  \ref{lcg}(3). %предложение 4.3.1.3(III.1) статьи \cite{bos}). 
 Moreover, the category  $\dmer$ is symmetric monoidal  (see  \cite[\S2.3]{bev}). 

Now let us give some alternative descriptions of $\dmer$; these are similar to certain constructions in the literature. 
We will need some more notation.

$\sscr$ will denote the full subcategory of $\pscr$ whose objects are those functors that yield Nisnevich sheaves (on the \'etale site of $\sv$; the objects of $\sscr$ can be called  $R$-linear {\it sheaves with transfers}). We recall that this category is abelian also (this is an easy consequence of  \cite[Theorem 13.1]{vbook}), and all $\rtr(X)$ are its objects  (easy from Lemma 6.2 of ibid.). %"плоское" тензорное произведение с "целого" случая????

We will also need the additive category  $\smcrpl$ of all coproducts of sheaves of the type   $\rtr(X)$ in the category $\sscr$ (certainly, $\smcrpl$ is also a full subcategory of  $\pscr$). %\pscr?? легко можно описать??
We will write $K'(\smcrpl )$ %полную 
for the localizing subcategory of  the homotopy category $K(\smcrpl)$ that is generated by  $\obj \smcrpl$ (here we consider objects of  $\smcrpl$ as one-term complexes; note that $K(\smcrpl)$ is closed with respect to coproducts). 

\begin{pr}\label{pisomot}
The following categories are canonically equivalent:

1) The localization of $K'(\smcrpl )$ by the localizing subcategory  $\du$ generated by $HI\cup \zmv$; %комплексами типа HI и \zmv;

2) $\dmer$;

3)  the localization $\dmer'$ of the category $D(\sscr )$ by the localizing subcategory generated by $HI$.

\end{pr}
\begin{proof}
All the objects of the form  $\rtr(X)$ (for $X\in \sv$) as well as bounded complexes whose terms are of this type are compact in the categories  $K'(\smcrpl )\subset K(\smcrpl )\subset K(\pscr)$ (and also is $D(\pscr)$). Denote the triangulated subcategory  $\lan \rtr(X):\ X\in \sv\ra\subset K'(\smcrpl )$ by $K^f(\smcrpl )$.
Applying Lemma  \ref{lcg}(3) we obtain that the natural functor  $K'(\smcrpl )/\du\to \dmer$ %yields an equivalence %задает изоморфизм??
gives an isomorphism between those full  subcategories whose objects come from  $K^f(\smcrpl )$. Applying part  2 of the lemma we obtain that the category  $K'(\smcrpl )/\du$ is equivalent to $ \dmer$.

%Чтобы доказать, что 
 Now we prove that  $ \dmer$ is equivalent to $\dmer'$. Note that the natural functor $D(\pscr)\to D(\sscr)$ respects coproducts and kills all elements of $\zmv$; hence there exists an exact functor $\dmer\to \dmer'$ that respects coproducts. Recall also that the Nisnevich sheafifications of  objects of $\pscr$ (considered as presheaves on the aforementioned site) are objects of $\sscr$ (i.e., sheaves with transfers); this statement follows from \cite[Theorem 13.1]{vbook} as well. It obviously follows that the category  $ D(\sscr)$ is equivalent to the localization of  $D(\pscr)$ by the localizing category generated by those presheaves whose sheafification is zero. Hence it suffices to verify that the localizing subcategory of  $D(\pscr)$ generated by $HI\cup \zmv$ contains all presheaves of this sort. 
This fact can be easily justified using an argument used in the proof of \cite[Theorem 3.2.6]{1} (where it was established in the case  $R=\z$).
\end{proof}

\begin{rema}\label{rdef}
\begin{enumerate}
\item\label{ir1}
Our list of descriptions of  $\dmer$ can be completed  %by the category obtained via
 by means of replacing $\zmv$ by a larger set of complexes. Recall that the Cartesian square   (\ref{enis}) is called an elementary distinguished  square if the morphisms $k$ and $j$ are open embeddings,  $f$ and $g$ are \'etale, and the base change of $f$ to  $X\setminus j(V)$  is an isomorphism. %также выполнено для...??
Denote by $\nmv$ the set of complexes  (\ref{enisc}) corresponding to squares satisfying these conditions; this set obviously contains  $\zmv$.  %каждый комплекс  типа \zmv также является комплексом типа \nmv. С другой стороны, комплексы типа \nmv являются объектами категории %$\du$, упомянутой в Предложении \ref{pisomot} 
Conversely, the category %$\lan HI\cup \zmv\ra_{K^f(\smcrpl )}$
 $\kar_{K^f(\smcrpl )}\lan HI\cup \zmv\ra$ contains  $\nmv$; indeed, it suffices to verify this statement in the case  $R=\z$, and 
then it easily follows from the aforementioned Theorem 3.2.6 of \cite{1}.

Thus in all the three descriptions in Proposition \ref{pisomot} the set  $\zmv$ may be replaced $\nmv$ (and the localizations will not change). This reduces the equivalence of  $\dmer$ with  $\dmer{}'$ to the following well-known fact:  the category $D(\sscr)$ is equivalent to the localization of  $ D(\pscr)$ by the localizing subcategory generated by  $\nmv$; see also \S6.2 of \cite{cdhom}.

\item\label{ir2}
 Let $X\in \sv$. The aforementioned statement implies that the sheaf $\rtr(X)$ is compact in  $D(\sscr)$.

We also give another proof of the latter fact. According to Exercise  13.5 of \cite{vbook} (see also Lemma  13.4 of loc. cit.), for each  $i\in \z$ and a bounded above complex $C$ of objects of  $\sscr$ the group $D(\sscr)(\rtr(X),C[i])$ is naturally isomorphic to  $ H^i_{Nis}(X,C)$ (i.e., to the $i$th Nisnevich hypercohomology group of  $X$ with coefficients in $C$). This fact immediately extends to all objects $D(\sscr)$. Since for any family     $(C_j)$ of objects of  $D(\sscr)$ we have $H^0(X,\coprod C_j)\cong \bigoplus _j H^0(X, C_j)$ (see Corollary   1.1.11 and \S1.1.12 of \cite{cdweil}), we obtain the compactness statement in question. 

\item\label{ir3} Hence the category $D(\sscr)$ is compactly generated. Since the localization functor  $L_R:D(\sscr)\to \dmer{}'$ respects coproducts (see Lemma \ref{lcg}(3)), part 1 of the lemma gives the existence of an adjoint functor  $i_R: \dmer{}'\to  D(\sscr)$ that is certainly a full embedding. We will often identify the categories  $\dmer$  and $\dmer{}'$ with the essential image of  $i_R$ (that is called {\it the category of motivic complexes}). We will use the notation  $\mgr(X)$ for  the image %images???
of $\rtr(X)$ (for $X\in \sv$) in all these categories.

It is easily seen (see Theorem 9.1.16 of \cite{neebook}) that motivic complexes are characterized inside the category $D(\sscr)$ by the following homotopy invariance conditions for the presheaves $ H^i_{Nis}(-,C)$: for all $i\in \z$ and  $X\in \sv$ we have $ H^i_{Nis}(X,C)\cong H^i_{Nis}(X\times \afo ,C)$. The functor $i_R\circ L_R$ can be described by an explicit formula  (see \cite[Remark 14.7]{vbook}); yet we will not need this fact below. 
%однако, мы не будем пользоваться этим утверждением и, поэтому, не будем его формулировать и обосновывать. 

We note also that the adjunction between  $L_R$ and $i_R$ combined with the isomorphism mentioned in part  \ref{ir2} of this remark implies that the group  $\dmer(\mgr(X),C[i])$ is naturally isomorphic to $ H^i_{Nis}(X,C)$ for each motivic complex  $C$ and $i\in \z$.

\item\label{ir4}
One of the advantages of our first description of  $\dmer$ is that it simplifies checking the compactness of all $\mgr(X)$ in this category. Moreover, we will use some more properties of  $\dmer$ established in \S6 of \cite{bev}. Still we note that in the case  $R=\z$ all these statements were proved in the papers of V. Voevodsky and F. D\'eglise; to generalize them to the case of an arbitrary   $R$ one may apply Proposition  \ref{pcmot} below. 

In particular, below we will apply the following property of  $\dmer$: for each smooth $Y/k$ and smooth proper  $X/k$ all of whose connected components are of dimension  $n$ and $i\ge 0$ the group  $\dmer(\rtr(Y),\rtr(X)[i])$ vanishes if  $i>0$ and equals  $CH^n(X\times Y)\otimes_{\z} R$ if $i=0$; see  \cite[Corollary 6.7.3]{bev}. Recalling also that the composition of morphisms in the full subcategory   $\operatorname{Corr^{rat}_R}$ of $\dmer$ whose objects are all $\rtr(X)$, is compatible with the composition of morphisms in the category of effective Chow motives (note that it suffices to prove this statement in the case $R=\z$; see Proposition \ref{pcmot} below) we obtain the following: the additive category  $\chower=\kar_{\dmer} \operatorname{Corr^{rat}_R}$ is the natural  $R$-linear version of effective Chow motives.\footnote{Note that the category  $\dmer$ is Karoubian according to Proposition 1.6.8 of \cite{neebook}; hence  $\chower$ is Karoubian also.}

\end{enumerate}
\end{rema}

We will also need some properties of the "extension of scalars" for motivic complexes.

\begin{pr}\label{pcmot}
 Let $R'$ be an associative commutative unital  $R$-algebra. Then the natural functor  $\otimes_R{R'}: D(\sscr)\to D(\sscrp)$ %flat resolutions???!
 %обладает точным правым сопряженным и?????
  yields a commutative diagram 
\begin{equation}\label{ecmot}
\begin{CD}
 \smc@>{\mgr }>>\dmer' @>{i_R}>>D(\sscr)  @>{L_R}>>\dmer'\\
@VV{=}V@VV{-\otimes^{mot}_R{R'}}V@VV{-\otimes_R{R'}}V @VV{\otimes^{mot}_R{R'}}V \\
\smc@>{\mgrp}>>\dmerpp@>{i_{R'}}>>D(\sscrp) @>{L_{R'}}>>\dmerpp
\end{CD}
\end{equation}
of functors.
\end{pr}
\begin{proof}
The existence (and commutativity) of the right hand square in our diagram is obvious. 

Next, recall that motivic complexes are characterized by the homotopy invariance of the presheaves $ H^i_{Nis}(-,C)$ (see Remark \ref{rdef}(\ref{ir3})). Combining this fact with Theorem 22.3 of \cite{vbook} (that says that the Nisnevich sheafification respects the homotopy invariance of presheaves with transfers) we obtain that 
 the functor $ -\otimes_R{R'}$ sends $R$-linear motivic complexes into $R'$-linear ones. 
Combining the latter statement with  \cite[Proposition 1.1.1(III)]{bpure}  % the properties of Bousfield localizations of triangulated categories  (see 
(that treats Bousfiled localizations of triangulated categories following \cite[\S9]{neebook}; cf. also Remark 1.3.3(3) of ibid.) % and the distinguished triangle  (1.1.1) in \cite{bpure}) 
 we easily obtain the commutativity of the middle square in the diagram. %  (since the functor  $-\otimes^{mot}_R{R'}$  sends $R$-linear motivic complexes  into $R'$-linear ones; see Remark \ref{rdef}(\ref{ir3})).

  It remains to note that the commutativity of the left hand square in the diagram follows immediately from the commutativity of the diagram obtained from our one by means of deleting the second column (certainly, the horizontal arrows passing through it should be composed pairwisely).
\end{proof}

\begin{rema}\label{rcoeff}
1. Obviously the functor  $-\otimes_R{R'}$ respects coproducts. Since  $L_R$ is surjective on objects, and both $L_R$ and $L_{R'}$ respect coproducts (see Remark  \ref{rdef}(\ref{ir3})), we obtain that the functor  $-\otimes^{mot}_R{R'}$ respects coproducts also. 

The functors $i_R$ and $i_{R'}$ respect coproducts as well  (see Remark \ref{rdef}(\ref{ir3}) once again), but we will not need this fact. 

2. The cases  $R=\z$ and $R=\zop$ appear to be the most interesting in the context of this paper. If one restricts to these cases then the corresponding motivic extension of scalars functor can be described as a certain Verdier localization; see appendix A of \cite{kel}.  %simplifies the proofs????
%В этих случаях функторы ""
\end{rema}

\section{On the Chow weight structure for effective motives}% ных комплексов}
\label{swchows}

In \S\ref{sdeff} we give the definition of our "effective" Chow weight structure $\wchows$ and discuss its relation to the Chow weight structures defined in previous papers.
Our description of $\wchows$ implies that a motivic complex has non-negative weights if and only if its positive Nisnevich hypercohomology vanishes. 

In \S\ref{stwist} %a theorem relating 
 the relation of our weight structure $\wchows$  to motivic twists is studied; this gives bounds on weights of certain motives. We also prove that the twist functor $-\lan 1\ra=-(1)[2]$ is weight-exact; this statement is important for the next section.

Since the distinctions between  $\dmer$, $\dmer'$, and the category of motivic complexes (see Remark \ref{rdef}(\ref{ir3})) are irrelevant for our arguments below, we will use the notation $\dmer$ to the category of motivic complexes.

\subsection{The definition of %и основные свойства 
$\wchows$ and its comparison with the Chow weight structures defined earlier }\label{sdeff}

We will use the notation  $w_{\chower}$ for the weight structures generated by the set $\mgr(\sv)$ in the category  $\dmer$ (see Proposition  \ref{pgenw}(II.1)). We will write just $\wchows$ for it and call it the Chow weight structure when this will cause no ambiguity.

\begin{rema}\label{robv}
The definition of  $\wchows$ along with the Remark \ref{rdef}(\ref{ir3}) obviously imply that the motivic complex  $C$ belongs to $ \dmer_{\wchows\ge 0}$ if and only if for any $ X \in\sv$ and $i>0$ we have $ H^{i}_{Nis}(X,C)=\ns$.

Applying Proposition  \ref{pgenw}(I.1) we also obtain that $\mgr(\sv)\subset \dmer_{\wchows\le 0}$.
\end{rema}

We will use the notation   $\spv$ for the set of smooth projective $k$-varieties. %???и докажем важную теорему о сравнении весовых структур.

\begin{theo}\label{tspv}
1. $\chower\subset \underline{Hw}_{\chower}$

2. Assume that $R$ is a $\zop$-algebra. Then  $w_{\chower}$ is also generated by  the set $\mgr(\spv)$. %неположительными  сдвигами 
 %(т.е., $R$-линейными мотивами гладких {\bf проективных} многообразий).\end{theo}

3.  The functor $-\otimes^{mot}_R{R'}$ (as defined in  Proposition  \ref{pcmot}) is weight-exact (with respect to the weight structures  $w_{\chower}$ and $w_{\chowerp}$). %R"=??
\end{theo}
\begin{proof} %{\bf Доказательство.}

1. Since the category  $\chower$ is Karoubi-closed in  $\dmer$, it suffices to verify that $\mgr(\spv)\subset \dmer_{\wchows=0}$.

We have already noted that  $\mgr(\sv)\subset \dmer_{\wchows\le 0}$. Hence it remains to prove that  $\mgr(\sv)\perp \mgr(\spv)[i]$ for all $i>0$.
The latter fact follows from Remark  \ref{rdef}(\ref{ir4}) immediately.

2. According to Proposition \ref{pgenw}(I.1) it suffices to verify that  
$$(\cup_{i<0} \mgr(\spv)[i])\perpp=(\cup_{i<0} \mgr(\sv)[i])\perpp.$$ Obviously, the second of these classes is contained in the first one. To check the inverse inclusion it suffices to prove that $\mgr(X)$ belongs to the envelope of the set $\cup_{i\le 0}\mgr(\spv)[i]$  for each  $X\in\sv$.
From the commutativity of the left hand square in  (\ref{ecmot}) it follows that it %is sufficient
 suffices to verify the latter statement in the case $R=\zop$ (recall that in the case $p=0$ the symbol  $\zop$ denotes the ring $\z$). In this case the fact follows immediately Theorem  2.2.1(3) and Proposition  1.3.2(iii) of \cite{bzp}.\footnote{It was assumed in ibid. that  $p>0$; still all the arguments of that paper can be applied in the case  $p=0$ also. Moreover, if  $p=0$ then one can apply Theorem 6.2.1(1) of \cite{mymot}.}
%Очевидно, достаточно проверить, что  

%$\mglam(X_i)[-i]$, где $X_i\in\spv$ (Rem 2.1.5 (3) из [Bon16]). %Не умаляя общности, можем 

3. Proposition  \ref{pgenw}(II.2)  obviously implies that  the functor $-\otimes^{mot}_R{R'}$ is left weight-exact. 

To prove that  $-\otimes^{mot}_R{R'}$  is right weight-exact we should verify the following: if for a motivic complex $C$ we have  $ H^{i}_{Nis}(X,C)=\ns$  for all $ X \in\sv$ and $i<0$, then the same condition is also fulfilled for  $C\otimes^{mot}_R{R'}\in \obj D(\sscrp)$ (see Proposition \ref{pcmot}). Certainly, to check this vanishing statement one can consider  $C\otimes^{mot}_R{R'}$ as an object of $D(\sscr)$ %.\footnote{Т
(so here we apply to  $C$ the forgetful functor  $F:D(\sscrp)\to D(\sscr)$, that is right adjoint to the functor $-\otimes_R{R'}$).  %}
 Next, the complex $F(C\otimes^{mot}_R{R'})$ can be obtained by tensoring  $C$ by a flat  $R$-module resolution of  $R'$; since the latter is concentrated in non-positive degrees, this gives the result in question  (cf. \cite[Definition 14.2, \S8]{vbook}). \footnote{We certainly can assume that the complex  $C$ has zero terms in positive degrees; this allows us to apply the results of ibid. Moreover, this calculation can also be easily made using other descriptions of  $\dmer$ given above.} %[\S8??]{vbook}?? Квазиизоморфен ограниченному сверху? Конечно же,  %спектральная последовательность?? , что из 
%посчитать руками через сопряженность; тензорное произведение из книжки??
\end{proof}

\begin{rema}\label{rnewoldchow}
1. Since $\chower\subset  \underline{Hw}_{\chower}$, the subcategory  $\chower$ is {\it negative} in $\dmer$, i.e., $\obj \chower\perp \obj \chower[i]$ for all $i>0$. Since  $\chower$ is Karoubian, Corollary  2.1.2 of \cite{bonspkar} %implies that there exists 
 gives a unique weight structure $w$ on $\lan \obj \chower\ra$ such that  $\chower\subset \hw$. Moreover, $\chower= \hw$,  the class $\lan \obj \chower\ra_{w\le 0}$ coincides with the envelope of $\cup_{i\le 0}\obj \chower[i]$, and the class  $\lan \obj \chower\ra_{w\ge 0}$ equals the envelope of $\cup_{i\ge 0} \obj \chower[i]$. % $i\ge 0$. 
%для этой весовой структуры 

In the settings considered in  \cite{bws} and \cite{bzp} it is known that $\lan \obj \chower\ra$ coincides with the subcategory $\dmger$ of compact objects of $\dmer$ ($\dmger$ is called the category of  {\it effective geometric motives}). Respectively, this weight structure  $w$ was called the Chow one.

2. Now, for any weight structure $(\cu,w)$ the class $\cu_{w\le 0}$ is closed with respect to all $\cu$-coproducts (of its objects).  Since $\wchows$ is generated by a set of compact objects of  $\dmer$, the class $\dmer_{\wchows\ge 0}$ is also closed with respect to  $\dmer$-coproducts (weight structures of this type are called {\it smashing} ones).

Next, take the localizing  subcategory $\dmer^{\chower}$ generated by  $\obj \chower$ in   $\dmer$, and consider the weight structure  $w'_{\chower}$ generated by  $\obj \chower$ in $\dmer^{\chower}$ (note that we can assume that  $\obj \chower$ is a set, so that we can apply Proposition  \ref{pgenw}(II.1)). 
We immediately obtain that the embedding  $\lan \obj \chower\ra\to \dmer^{\chower}$ is weight-exact with respect to the weight structures $w$ and $w'_{\chower}$. 

Moreover, Corollary 2.3.1(1) of \cite{bsnew} implies that  the embedding $\dmer^{\chower}\to \dmer$ is weight-exact (with respect to $w'_{\chower}$ and $\wchows$) as well.%\footnote{One can also prove this fact by noting that $w'_{\chower}$ is {\it class-generated} by $\obj \chower$ in the sense of Definition 1.2.2(8) of ibid., whereas the latter statement can be deduced from Theorem 2.1.1(I.1)  of \cite{bonspkar}.}   % $w_{\chower}$ can be described as follows: $\dmer^{\chower}_{w'_{\chower}\le 0}$  (resp., $\dmer^{\chower}_{w'_{\chower}\ge 0}$) is the smallest class of objects of $\dmer^{\chower}$ that contains $\obj \chower$ and is closed with respect to extensions, coproducts, and also with respect to $[-1]$ (resp., with  respect to $[1]$).\footnote{This statement can be deduced from Theorem 2.1.1(I.1)  of \cite{bosn}, and it follows immediately from % cf. also \S2.2 
 %Corollary 2.3.1(1) of \cite{bsnew} (to apply %the latter 
 %loc. cit. one does not have to assume that  the class $\obj\chower$ is a set).} Hence the embedding $\dmer^{\chower}\to \dmer$ is weight-exact (with respect to $w'_{\chower}$ and $\wchows$) as well.  
%Кроме того, из предложения \ref{pgenw}(I.1) легко следует, что   $w_{\chower}$ --- единственная весовая структура конечного типа  на $\dmer$, ядро которой содержит $\chower$ (cf. Theorem 2.2.1(5) of \cite{bsnew}). % (также см. \cite{bnsurv}).
  Applying Proposition  \ref{pgenw}(I.1) one can easily deduce that  $\lan \obj \chower\ra_{w\le 0}=\dmer_{\wchows\le 0}\cap \obj \lan \obj \chower\ra$, $\lan \obj \chower\ra_{w\ge 0}=\dmer_{\wchows\ge 0}\cap \obj \lan \obj \chower\ra$, 
$\dmer^{\chower}_{w'_{\chower}\le 0} =\dmer_{\wchows\le 0}\cap \obj \dmer^{\chower}$, and $\dmer^{\chower}_{w'_{\chower}\ge 0}=\dmer_{\wchows\ge 0}\cap \obj \dmer^{\chower}$ (see Proposition 1.2.5(1) of ibid.).  
%$C_1=\dmer_{\wchows\le 0}\cap \obj \dmer^{\chower}$, а $C_2=\dmer_{\wchows\ge 0}\cap \obj \dmer^{\chower}$. 

%Таким образом, 
 All these results demonstrate that  $\wchows$  is "closely related"  to weight structures "generated" (in the corresponding senses) by Chow motives; %so we will also say that 
 this is why we call $\wchows$  a Chow weight structure. Note also that to prove that   $\wchows=w'_{\chower}$ it suffices to verify that  $\dmer^{\chower}= \dmer$; %этот факт равносилен 
  the latter equality is equivalent to $\lan \obj \chower\ra=\dmger$.

3. The main disadvantage of  $\wchows$ (for a general $R$ and $p>0$)  is that we cannot describe its heart explicitly, and do not know whether this weight structure restricts to  $\dmger$.
%+ no motives of singular varieties???

Recall also that for any functor  $H$ from $\dmer$ into an abelian category  the weight structure $\wchows$ gives a certain filtration on the values of $H$ (see \cite[Proposition 2.1.2]{bws}), that can be called the Chow-weight one  (in \cite[Remark 2.4.3]{bws} and \cite[Remark 2.4.5]{bpure} it is explained that Chow-weight filtrations vastly generalized Deligne's weight filtrations on singular and \'etale cohomology of varieties); these filtrations are  $\dmer$-functorial. Thus our %version of ??
 Chow weight structure gives interesting weight filtrations for any  $R$ (in particular, we can consider them for  "$p$-adic" functors).

If $H$ is homological or cohomological then the theory of weight structures developed in \S2.3--2.4 of \cite{bws} yields a relation of cohomology of arbitrary motivic complexes to that of objects of the heart of $\wchows$. The disadvantages of  $\wchows$ described above make the application of this result rather difficult; yet cf. Proposition \ref{popen}(1) below.

4. Theorem \ref{tspv}(2) is the only statement in this paper whose proof relies on certain resolution of singularities results for  $k$-varieties; note that these statements were crucial for the corresponding results of  \cite{bws} and \cite{bzp}. %Это, возможно, свидетельствует о том, что методы данной статьи 
 This hints that the methods of the current paper can be applied to categories of relative motives (over a base scheme distinct from the spectrum of a field). Recall also that a certain compactly generated Chow weight structure has found interesting applications to the so-called relative $K$-motives in \S2.3 of  \cite{bkl}.
\end{rema}

\subsection{On twists and their weight-exactness}\label{stwist} 

Now let us recall the notion of (Tate) twists  $-\lan n \ra$ and $-(n)$, and Gysin distinguished triangles; these are important for motives. 

Since the identity of the point $\spe k$ factors through  $\p^1$, we have  $\mgr(\p^1)\cong R\bigoplus R\lan 1 \ra$, where $R$ is the motif of the point and $R\lan 1 \ra$ is the motif that Voevodsky called the Tate one (though calling it the Lefschetz motif would may be more appropriated)  and denoted by $R(1)[2]$. Certainly $R\lan 1\ra\in \obj \chower$. 

Since $\mgr(X)\otimes \mgr(Y)\cong \mgr(X\times Y)$ for any $X,Y\in \sv$ (here we use the tensor product on  $\dmer$),
the subcategory  $\chower\subset \dmer$ is closed with respect to the tensor product. Hence the  $n$th tensor power  $R\lan n\ra$ of $R\lan 1\ra$ (that can also be denoted by  $R(n)[2n]$) is an object of $\chower$ for each  $n\ge 0$.  We will use the notation $M\lan n\ra$ for  $M\otimes  R\lan n\ra$ for any $M\in \obj\dmer$.
The properties of the tensor product of  $\dmer$ easily imply that  the functor  $-\lan n \ra$ respects coproducts and the compactness of objects.\footnote{Since  the category $\dmer$ is compactly generated by the set $\mgr(\sv)$, its subcategory $\dmger$ of compact objects coincides with $\lan \mgr(\sv)\ra$. It follows immediately that the functor  $-\lan n \ra$ respects compactness. However, below we will only apply the compactness of the %objects of the form сохраняет компактность. Впрочем, ниже нам понадобится только компактность элементов
 elements of $\mgr(\sv)\lan n \ra$.}  
 
We will need the following facts. 

\begin{lem}\label{lgys}
Assume that  $n\ge 0$.

1. Let %гладкое многообразие 
$Z\in \sv$  be an equicodimensional closed subvariety of codimension  $n$ in a smooth variety  $X$  %и 
(i.e.,  the connected components of $Z$ are of codimension  $n$ in  $X$).
Then there exists a Gysin distinguished triangle % (\cite[Определение 4.6]{deg}),
$$\mgr(X\setminus Z) \to \mgr(X) \to\mgr(Z)\lan n \ra \to \mgr(X\setminus Z)[1]$$
in $\dmer$.

2. The endo-functor  $-\lan n \ra$ is left weight-exact; in particular,  $\mgr(\sv)\lan n \ra\subset \dmer_{\wchows\le 0}$.
\end{lem}
\begin{proof}
1. This assertion is a re-formulation of  \cite[Proposition 6.3.1]{bev} (see also Proposition  4.3 of \cite{degysin}).

2.  Since the functor  $-\lan n \ra$ respects coproducts,  Proposition \ref{pgenw}(II.2) says that it suffices to verify whether $\mgr(\sv)\lan n \ra\subset \dmer_{\wchows\le 0}$. It remains to recall that for each  $X\in \sv$ the motif $\mgr(X)\lan n \ra$ is a retract of $\mgr(X\times \prod_{1}^n \p^1)$.
\end{proof}

To illustrate these properties of motives and twists we prove two easy statements on "weight bounds".

\begin{pr}\label{popen}
 
1. Assume that $U$ is a variety of the form  $X\setminus \cup_{i=1}^n Z_i$, where $X$, all of $Z_i$, and all the intersections of  subfamilies of  $Z_i$ are smooth proper  $k$-varieties; suppose moreover that the intersections of any  $m$ of distinct $Z_i$'s is empty (for some $m\le n\in \z$).

Then the motif  $\mgr(U)$ belongs to the envelope of  $\cup_{i=0}^m\obj\chower[i] $,  and so, to $\obj (\lan \obj \chower\ra) \cap \dmer_{[0,m]}$. %=ext-closure of \chower[i]??

2. Let $f: U\to V$ be an open dense embedding,  where  $U,V \in\sv$. Then  $\co(\mgr(f)) \in\dmer{}_{\wchows\le 0}$.
\end{pr}

\begin{proof} %В доказательстве теоремы \ref{tspv} мы показали, что если $U\in\sv$, $dimU=d$, то $\mg(U)$ лежит в оболочке $\{\mg(P)[-r], P\in{SmProjVar}, 0\le r\le d\}.$
1. Since $\chower\subset \underline{Hw}_{\chower}$ and the class $\dmer_{[0,m]}=\dmer_{\wchows\ge 0}\cap \dmer_{\wchows\le m}$ is closed with respect to extensions and retractions, it suffices to verify that $\mgr(U)$ belongs to the envelope of $\cup_{i=0}^m\obj\chower[i] $.
%$\dmer^{\chower}$   
We prove this statement by induction on  $n$. In the case  $n=1$ it is obvious. %Е
%Так как операция дизъюнктного объединения многообразий соответствует прямой сумме их мотивов, мы можем считать, что $X$ и все $Z_i$ связны. %несвязные пересечения?? 

Assume now that the statement is fulfilled for all $n'<n$. 
We present $U$ as $(X\setminus \cup_{i=2}^n Z_i)\setminus (Z_1\setminus \cup_{i=2}^n Z_i)$. Denote the variety  $X\setminus \cup_{i=2}^n Z_i$ by $X'$, and the connected components of  $Z_1\setminus \cup_{i=2}^n Z_i$ by $Y_j$ (here $1\le j\le l$ for some $l\in \z$).

According to our inductive assumption, the motif  $\mgr(X')$ belongs to the envelope of  $\cup_{i=0}^m\obj\chower[i] $, and all $\mgr(Y_j)$ belong to the envelope of $\cup_{i=0}^{m-1}\obj\chower[i] $. % (считаем, что $1\le j\le l$ для некоторого $l\in \z$). 
 Denote the codimensions of  $Y_j$ in  $X$ by $c_j$. Since $\obj \chower\lan c_j\ra\subset \obj \chower$, all $\mgr(Y_j)\lan c_j\ra$ also belong to the envelope of  $\cup_{i=0}^{m-1}\obj\chower[i] $.

Next, the Gysin distinguished triangle  (see Lemma \ref{lgys}(1)) gives distinguished triangles $$\mgr(Y_{r+1}) \lan c_{r+1}\ra[-1]\to  \mgr(X'\setminus(\sqcup_{j=1}^{r+1})Y_j)\to \mgr(X'\setminus(\sqcup_{j=1}^{r})Y_j)\to \mgr(Y_{r+1})\lan c_{r+1}\ra $$ for all $r$ between $0$ and $l-1$. Since $U=X'\setminus(\sqcup_{j=1}^{l})Y_j$,  these triangles yield our assertion.

%искомое утверждение выполнено для  
%Так как $\chower\subset \hw$. сли индукwионное предположение выполненоТаким образом, достаточно заметить, что если $m>1$ и $\mgr(U_1)\in \obj \lan \chower\ra \cap \dmer_{[0,m]}$, то выделенный треугольник Гизина (см. \ref{lgys}(1)) легко дает результат (ср. с  доказательством части 2), так как класс $\obj \lan \chower\ra \cap \dmer_{[0,m]}$ замкнут относительно расширений. 

2. There obviously exists a sequence of open dense embeddings 
 $U_{0}=U\subset U_{1}\subset \dots \subset U_{m}=V$ (for some $m>0$) such that the varieties  $U_{i+1}\backslash U_{i} $ are smooth and equicodimensional in $U_{i+1} $ for all  $i$  between  $0$ and $m-1$.

The corresponding Gysin distinguished triangles along with the octahedral axiom of triangulated categories give distinguished triangles
$\mgr(U_{i+1}\backslash U_{i})\lan c_i\ra \to \co(\mgr(U_i)\to \mgr(V))\to \co(\mgr(U_{i+1}\to \mgr(V))\to \mgr(U_{i+1}\backslash U_{i})\lan c_i\ra[1]$ for all $i$ between $0$ and $m-1$, where  $c_{i}$ is the codimension of  $U_{i+1}\backslash U_{i} $  in $U_{i+1} $. 
Hence $\co(\mgr(f))$ belongs to the envelope of (all)  $\mgr(U_{i+1}\backslash U_{i})\lan c_i\ra$. Since the class  $\dmer{}_{\wchows\le 0}$ is extension-closed,  it remains to apply Lemma \ref{lgys}(2).
 
\end{proof}

Now let us prove a theorem that is crucial for the next section. 

\begin{theo}\label{twist}
The endo-functor  $-\lan n \ra$ is weight-exact for all  $n\ge 0$.
\end{theo}
\begin{proof}
Since this functor is left weight-exact (see Lemma \ref{lcg}(2)), it remains to verify that it is right weight-exact. Obviously it suffices to verify the latter statement in the case  $n=1$.

Now we argue somewhat similarly to the proof of Theorem \ref{tspv}(3). 

For a motivic complex $C$ such that   $ H^{i}_{Nis}(X,C)=\ns$ for all $ X \in\sv$ and $i>0$ we should check that these conditions are also fulfilled for $C\lan 1\ra$.

We note that the natural forgetful functor  $\operatorname{For}:\dmer\to \dmez$ commutes with the twist  $-\lan 1 \ra$ (this can be easily checked using the first of the descriptions of  $\dmer$ listed in Proposition \ref{pisomot}). %{vbook}???! %, так как 
%; это легко следует из аналогичного  %??! %свойства   %пояснить? Почему?!!!!
Moreover, the group $ H^{i}_{Nis}(X,C)$ is certainly isomorphic to  $ H^{i}_{Nis}(X,\operatorname{For}(C))$ for all $X\in \sv$ and $i\in \z$.  Thus we can assume  $R=\z$; we fix some $C\in \dmez_{\wchows\ge 0}$. %Это позволяет нам легко вывести искомое утверждение из следующих лемм. 

We consider the following cohomology theories on smooth $k$-varieties:  for each $r\ge 0$, $q\in \z$,  and  %гладкого многообразия $X/k$ 
 $X\in \sv$ we will use the notation $H^q_r(X)$ for the group $\dmez(\mgz(X)\lan r\ra [-r-i], C\lan 1\ra )$.  
We should verify that  $H^{i}_0(X)=\ns$ for all $i>0$ and $X\in \sv$.
For this purposes we consider the following converging (coniveau) spectral sequence:
%Then the coniveau spectral sequence converging to $H^*_0$ has the form 
$$E_{1}^{r,q}=\coprod\limits_{x\in X^{(r)}}H_r^{q}(x) %,F(-p))
 \Rightarrow H_0^{r+q}(X),$$ where $X^{(r)}$ is the set of points of $X$ of codimension $r$, and for
a presentation of $x\in X^{(r)}$ %as in the right hand side 
 as $\prli X_j$ for $X_j\in \sv$ we define $H^*_*(x)$ %(\spe K)$
  as $\inli_j H^*_*(X_j)$. %Indeed, this fact easily follows from 
	This spectral sequence and its convergence is given by Proposition 4.3.1(I.3) of \cite{bgn} (see Remark 4.3.2(2) of loc. cit.); %\footnote{To apply loc. cit. one should either note that the functors %$\dmez(-, C\lan 1\ra):\dmez\to \ab$ factors through the category
	%$H^*_*$ factor through the functor $X\mapsto \om(X_+)$ into the motivic homotopy category $\sh$, or use the $\dmez$-version of loc. cit. provided by Proposition 5.2.6 of ibid.}
	see also \cite[\S6.3.2]{degmg}.
 \begin{comment}$H^i(X, -r)$  for the group $\dmez(\mgz(X)\lan r\ra [-2r-i], C\lan 1\ra )$;\footnote{We follow the notation of \cite[\S6.3]{degmg}; note that our notation  $\mgz(X)\lan r\ra [-2r]$ corresponds to $\mg(X)(r)$ in the notation of ibid.} for a function field $K/k$ such that  $\spe K=\prli X_j$ for $X_j\in \sv$ we define $H^i(K,-r)$ %(\spe K)$
  as $\inli_j H^i(X_j,-r)$.
For this purposes we consider the following converging spectral sequence from \cite[\S6.3.2]{degmg}: %\footnote{Наши обозначения несколько отличаются от тех, которые вводятся в начале параграфа 6.3 этой статьи).} 
$$E_{1}^{r,q}=\coprod\limits_{x\in X^{(r)}}H^{q-r}(k(x),-r) %,F(-p))
 \Rightarrow H^{r+q}(X,0),$$
where $X^{(r)}$ is the set of points of $X$ of codimension $r$. 
\end{comment} 
	
Now denote by  $t$  the unbounded version of the Voevodsky homotopy $t$-structure (see \cite[\S5.1]{deg}). According to Lemma  \ref{lwt} below, we have $C\in \dmez^{t\le 0}$; hence  $C\lan 1\ra\in   \dmez^{t\le -1}$. Since the spectra of function fields are %Nisnevich points, Так как спектры полей   функций являются %точками в топологии Нисневича
  the henselizations of the corresponding smooth varieties at their generic points, we obtain that  $E_{1}^{0,q}=\ns$ for $q\ge 0$. %нужное утверждение для нулевого сдвига немедленно получается из 
  Next, for a function field  $K/k$, %, where $\spe K=\prli X_j$ is 
	a presentation of  $\spe K$ as an inverse limit of smooth varieties $X_j$,  and any $r>0$ we have
$$\begin{gathered} H^{q}_r(\spe K)%=\coprod\limits_{x\in X^{(1)}}
=\inli_j \dmez (\mgz(X_j)\lan r\ra, %[-p], 
C\lan 1\ra [r+q])  \\
\cong\inli_j \dmez (\mgz(X_j)\lan r-1\ra, %[-p], 
C[r+q]) %\simeq \coprod\limits_{x\in X^{(1)}}\dme(\mathbb{Z}, C[i])=\ns
;\end{gathered}$$ here we apply the Cancellation theorem  (this is Corollary  4.10 of \cite{voecans}). Since $\mgz(X_j)\lan r-1\ra \in \dmez_{\wchows\le 0}$  (see Lemma \ref{lgys}(2)) %является ретрактом $\mgz(X_j\times G_m^{p-1})$ %(где  $G_m=\afo\setminus \{0\}$) и ?? $\mgz(X_j)\lan p-1\ra [1-p]\in \dmez_{\wchows\le 0}$ %$(\cup_{i>0}\mgz(\sv)[-i])
 and $\dmez_{\wchows\le -1}\perp C$, we obtain that $E_{1}^{r,q}=\ns$ if $r+q>0$ (also in the case  $r>0$).
 %если $q>0$.  
 It obviously follows that   $H_0^{r+q}(X)=\ns $ if $r+q>0$, as desired. %обнуление $E_{1}^{p,q}=\ns$ при $p>0$ и всех  $q$ %всех %выше перечисленных членов спектральной последовательности 
 %дает искомое.

\end{proof}
%Reduce to $r=\z$ or assume this??
	%T?! $R=\z$?!  omit \z in the notation?? %Для доказательства теоремы докажем следующие леммы. %\newtheorem{Lem}{Лемма}

It remains to prove the following properties of  $t$.

\begin{lem}\label{lwt}
1.  A motivic complex  $C$ belongs to  $\dmez^{t\le 0} $ if and only if for each  scheme $S$ that is the henselization of a variety  $X\in \sv$  in some point  and any presentation of  $S $ as $\prli S_j$, where $S_j\in \sv$, the group $\inli_j H^i_{Nis}(S_j,C)$ vanishes if $i>0$.

 %гладкие гензелевы (определить когомологии) или поля; ссылки?? 
2. $\dmez^{t\le 0}\lan 1 \ra \subset \dmez^{t\le -1} $.

3. $\dmez_{\wchows\ge 0}\subset \dmez^{t\le 0} $. %R-linear version for the "generated" $t$-structure?? + residues at fields?!! apply the forgetful functor?? what about ss?? 
\end{lem}
\begin{proof}
 
1.  According to Corollary  5.2 of \cite{deg}, $C$ belongs to $\dmez^{t\le 0} $ if and only if for each  $i>0$ the Nisnevich sheafification of the presheaf  $H^i(-,C)$ is zero. Hence it suffices to recall that the Nisnevich sheafification of a presheaf of abelian groups $P$ %with transfers
  (on the \'etale site of $\sv$) vanishes if and only if  $\inli_j P(S_j)=\ns$ for all projective systems $(S_j)$ as in the assertion. 

 %вспомнить замечание \ref{robv}. %применить предложение \ref{pcohom}.

2. Easy from Theorem 5.3 of loc. cit. 

3. If  $C$ belongs to   $\dmez_{\wchows\le 0} $, $i>0$, and   $S_j\in \sv$ for all $j$, then the group  $\inli_j H^i_{Nis}(S_j,C)$ vanishes. Hence the direct limits in part 1 of this lemma vanish also. 
\end{proof}

\begin{rema}\label{rcoefft}
1. Most probably all parts of the lemma are fulfilled for the natural $R$-linear version of $t$ (for any $R$). This statement is especially easy to verify in the case where  $R$ is a localization of $\z$ (i.e., $R\subset \q$); see Remark \ref{rcoeff}(2) and \cite[Proposition 5.6.2(II.3)]{bpure}.

2. Now recall that the functor $-\lan 1 \ra$ is fully faithful  (the $R$-linear version of this statement was established in \cite[\S6.1]{bev}).
We will assume that this functor is a full embedding and denote its essential image by $\dmer\lan 1 \ra$.

Combining our theorem with Proposition \ref{pgenw}(I.1) one can easily prove the following "cancellation theorem" for $\wchows$: $\dmer_{\wchows\le 0}\cap \obj (\dmer\lan 1 \ra)=\dmer_{\wchows\le 0}\lan 1\ra$ and %up to isomorphisms??!!  
$\dmer_{\wchows\ge 0}\cap \obj (\dmer\lan 1 \ra)=\dmer_{\wchows\ge 0}\lan 1\ra$; hence $\dmer_{\wchows= 0}\cap \obj (\dmer\lan 1 \ra)=\dmer_{\wchows= 0}\lan 1\ra$.\footnote{%Нужно заметить, 
Note that the classes $\dmer_{\wchows\ge 0}\cap \obj (\dmer\lan 1 \ra)$ and $\dmer_{\wchows\ge 0}\cap \obj (\dmer\lan 1 \ra)$ give a weight structure on the category  $\dmer\lan 1 \ra$, and the components of this weight structure contain the corresponding components of the weight structure  $(\dmer_{\wchows\le 0}\lan 1\ra, \dmer_{\wchows\ge 0}\lan 1\ra)$; then %????
 it remains to apply  \cite[Lemma 1.3.8]{bws}.}
\end{rema}

%\subsection{Связь $\wchows$ и мотивной подкрутки}\label{stwist} 
\section{Chow weight structures for stable and birational motives}\label{swchowo}

Now we study the naturally defined Chow weight structures on the categories obtained from  $\dmer$ by means of the twist  $\lan 1 \ra$.

In \S\ref{swchow} we "extend" $\wchows$ to the "stable" motivic category  $\dmr$ (on which the functor  $\lan 1 \ra$ is invertible); respectively, the embedding  $\dmer\to \dmr$ is weight-exact. 

In \ref{sbir} we prove that  $\wchows$ induces a weight structure $\wbir$ on the category $\dmer/\dmer\lan 1\ra$ of birational motives; this weight structure can be easily described. %???

\subsection{The weight structure  $\wchow$ on $\dmr$}\label{swchow}

The full faithfulness of the functor $-\lan 1 \ra$ yields the existence of a triangulated category $\dmr$ that is closed with respect to coproducts, is equipped with a full embedding $i:\dmer\to \dmr$ that respects coproducts and the compactness of objects, and such that the functor $-\lan 1 \ra_{\dmr}=-\otimes_{\dmr} i(R\lan 1 \ra)$ is an auto-equivalence of  $\dmr$ (see \S4.2 of \cite{deg} and \S1.1 of \cite{cdint}).\footnote{Note that the "stable" motivic category defined in  \cite[\S6.1]{bev} is not closed with respect to coproducts; so, this category is not "really nice". However, it is easily seen that this category  (that can be denoted by  $\dmer[\lan -1\ra]$) %is a full subcategory of 
 can be naturally embedded into $\dmr$,   there exists an obvious Chow weight structure  on 	 $\dmer[\lan -1\ra]$, and the embeddings $\dmer\to\dmer[\lan -1\ra]\to \dmr$ %в эту категорию также 
are weight-exact (cf. Theorem \ref{tstab}).}

We assume that the functor   $-\otimes_{\dmr} i(R\lan 1 \ra)$ is invertible on $\dmr$, and for each  $n\in \z$ we will denote its  $n$th composition power by  $-\lan n \ra$. Certainly, these functors respect coproducts and the compactness of objects. We will assume that  $\dmer$ is a full subcategory of  $\dmr$; so we will write  $\mgr(X)$ instead of $i(\mgr(X))$. Recall also that the category  $\dmr$ is generated by the set  $\cup_{j\in \z}(\mgr(\sv)\lan j \ra)$ as its own localizing subcategory. 
 
 We define $\wchow=w_{\chow,R}$ as the weight structure generated by the set $\cup_{j\in \z}(\mgr(\sv)\lan j \ra)$.
 Applying Theorem  \ref{twist} we easily obtain the following statements. 

\begin{theo}\label{tstab}
1. The auto-equivalences  $-\lan n \ra$ are weight-exact with respect to  $\wchow$ (for all $n\in \z$).

2. The embedding $i:\dmer\to \dmr$ is weight-exact  (with respect to $\wchows$ and $\wchow$).

\end{theo}
\begin{proof}
1. Since the functor   $-\lan n \ra$ is an auto-equivalence of $\dmr$ that induces a bijection on the set of all  $\dmr$-retracts of elements of  $\cup_{j\in \z}(\mgr(\sv)\lan j \ra)$, and this set generates  $\wchow$ also, the functor $-\lan n \ra$ is left weight-exact.  Applying Proposition \ref{pgenw}(I.1) we also obtain that  $-\lan n \ra$ is right weight-exact.

2. Since  $i$ respects coproducts and sends  "generating objects" of $\wchows$ into that of $\wchow$, applying Proposition \ref{pgenw}(II.2) we obtain that  $i$ is left weight-exact.

To prove that $i$ is also right weight-exact we should check for any $C\in \obj \dmer$ that if  $\mgr(\sv)[s]\perp C$ for all  $s<0$ then  $\mgr(\sv)[s]\lan r \ra\perp C$ for all $s<0$ and  $r\in \z$. %The latter statement is obvious if  
If $r\ge 0$ then the statement follows from Lemma \ref{lgys}(2). If $r<0$ then for each  $X\in \sv$ we have  $\dmr(\mgr(X)[s]\lan r \ra,C)\cong \dmr(\mgr(X)[s],C\lan -r \ra)= \dmer(\mgr(X)[s],C\lan -r \ra)$. Since $C\lan -r \ra\in \dmer_{\wchows\ge 0}$ according to Theorem  \ref{twist}, and the weight structure  $\wchows$ is generated by $\mgr(\sv)$,  we obtain that $\mgr(X)[s]\perp C\lan -r \ra$ indeed.  %мы получаем искомую ортогональность.

\end{proof}

\begin{rema}\label{rstablenewold} 
%Будем обозначать 
 Denote by $\chowr$  the full subcategory  of $\dmr$ whose object class equals $\cup_{i\in \z}\obj \chower\lan i \ra$; this category is obviously equivalent to the category of  $R$-linear Chow motives. Note also that Theorem \ref{tstab} implies that  $\chowr\subset \hw_{\chow}$.

Similarly to Remark \ref{rnewoldchow} we obtain that on the localizing subcategory $\dmr^{\chowr} $ of $ \dmr$ generated by  $\obj \chowr$ there exists a weight structure generated by $\obj\chowr$, and the embedding $\dmr^{\chowr}\to \dmr$ is weight-exact.  
\end{rema}

\subsection{The relation of $\wchows$ to the birational weight structure  }\label{sbir}

%Following??
 Applying Remark  \ref{rcoefft}(2) we consider the full triangulated subcategory $\dmer\lan1\ra$ of $\dmer$. In \cite{kabir} the Verdier localization  $\dmb=\dmer/\dmer\lan 1 \ra$ was considered (in the cases $R=\z$ and $R=\zop$). This category is called the category of birational motivic complexes; the reason for this is that the localization functor  $\pi:\dmer\to \dmb$ sends the motives of birationally equivalent smooth varieties into isomorphic objects (see Lemma \ref{lgys}(1)). The composition $\pi\circ \mgr$ will be denoted by $\mgb$.

\begin{pr}\label{pbir}
1. The elements of $\mgb(\sv)$ are compact in $\dmb$.

2. The functor  $\pi$ is weight-exact with respect to the weight structures  $\wchows$ and $\wbir$, where $\wbir$ is generated by the set  $\mgb(\sv)$.
%\cite{bws}.
\end{pr}
\begin{proof}
1. See Lemma \ref{lcg}(3).

%Since $\wchows$ restricts???
2. According to \cite[Proposition 8.1.1(1)]{bws} there exists a weight structure  on $\dmb$ such that the functor  $\pi$ is weight-exact. Applying Proposition \ref{pgenw}(II.3) we obtain that this weight structure is generated by  $\mgb(\sv)$.

\end{proof}

\begin{rema}\label{rbir} %+ Compatibility with "old" $\wchower$ if there exist smooth compactfications???!!!!
1. A significant distinction of $\wbir$  from  $\wchows$ in the context of this paper is that we can prove that $\wbir$  restricts to the subcategory $\dmgb=\lan \mgb(\sv)\ra $ of compact objects of $\dmb$. Indeed, Lemma \ref{lcg}(3) implies that %the latter category 
 $\dmgb$ is equivalent to  $\kar(\dmger/\dmger\lan 1\ra)$, and the existence on the latter category of a weight structure  generated by the subcategory  $\chowb$ %соответствующей
 (whose objects are retracts of elements of  $\mgb(\sv)$)  was established in    \cite[\S5.2]{bos}. The heart of this restricted weight structure equals    $\chowb$; %совпадает с $\kar_{\dmb}
%состоит из ретрактов элементов $\mgb(\sv)$; применив отсюда 
 thus  \cite[Corollary 2.3.1(1)]{bsnew} (as well as \cite[Theorem 4.5.2]{bws})  implies that the heart of $\wbir$ consists of all retracts of coproducts of elements of  $\mgb(\sv)$.

2. Arguing similarly to   \cite[\S5]{bos} one can prove that the category  $\dmb$ is equivalent to the localization of  $K'(\smcrpl )$ (see point  1 in Proposition  \ref{pisomot})  by the localizing subcategory, generated by cones of all $\mgr(f)$, %BIR??
  where $f:U\to X$ is a dense open embedding of smooth $k$-varieties. On the category  $K'(\smcrpl )$ we take the "stupid" weight structure, generated by  $\rtr(\sv)$; cf. \cite[Remark 2.3.2(1)]{bsnew} and \cite[Remark 1.2.3(1)]{bonspkar}. %  \cite[Remark 1.2.3(1)]{bonspkar} and  \cite[\S2.2]{bsnew}. %\cite{bnsurv}. %объектами %stupid; explain why??
Applying \cite[Theorem 4.3.1.4]{bos} (or  \cite[Theorem 3.1.3(3(ii))]{bsnew}) we obtain that there exists a weight structure on  $\dmb$ such that the localization  $K'(\smcrpl )\to \dmb$ is weight-exact. Hence Proposition  \ref{pgenw}(II.3) implies that this weight structure coincides with  $\wbir$ also.  Thus any element of  $\dmb_{\wbir\le 0}$ (resp. of $\dmb_{\wbir\ge 0}$)  is a retract of a object that has a pre-image in $K'(\smcrpl )$ that is presented by a  $\smcrpl$-complex  concentrated in non-negative (resp. non-positive) degrees (see Proposition 3.1.1(1) of \cite{bsnew}). %lift weights?! How?!
\end{rema}

%unique with these restrictions and the birational version; slices, {bnsurv}?? later; cf. {bger}??! 

 \end{document}